\documentclass[10pt]{amsart}
\usepackage[T1]{fontenc}
\usepackage[margin=1.2in,marginparsep=0.1in,marginparwidth=1in]{geometry}
\usepackage{amssymb,amsmath,amsthm,amstext,amscd,latexsym,graphics,graphicx,bbm,caption}
\usepackage[usenames,dvipsnames,svgnames,table]{xcolor}
\usepackage[plainpages=false,colorlinks=true, pagebackref]{hyperref}
\usepackage{tikz}
\usetikzlibrary{cd,positioning}

\tikzset{>=stealth}
\hypersetup{citecolor=Sepia,linkcolor=blue, urlcolor=blue}

\usepackage{cite}   

\makeatletter
\def\@tocline#1#2#3#4#5#6#7{\relax
  \ifnum #1>\c@tocdepth 
  \else
    \par \addpenalty\@secpenalty\addvspace{#2}%
    \begingroup \hyphenpenalty\@M
    \@ifempty{#4}{%
      \@tempdima\csname r@tocindent\number#1\endcsname\relax
    }{%
      \@tempdima#4\relax
    }%
    \parindent\z@ \leftskip#3\relax \advance\leftskip\@tempdima\relax
    \rightskip\@pnumwidth plus4em \parfillskip-\@pnumwidth
    #5\leavevmode\hskip-\@tempdima
      \ifcase #1
       \or\or \hskip 2em \or \hskip 2em \else \hskip 3em \fi%
      #6\nobreak\relax
    \dotfill\hbox to\@pnumwidth{\@tocpagenum{#7}}\par
    \nobreak
    \endgroup
  \fi}
\makeatother

\newtheorem{intro-thm}{Theorem}[]
\theoremstyle{plain}
\newtheorem{thm}{Theorem}[section]
\newtheorem{theorem}[thm]{Theorem}

\newtheorem{lemma}[thm]{Lemma}
\newtheorem{corollary}[thm]{Corollary}
\newtheorem{proposition}[thm]{Proposition}

\theoremstyle{definition}
\newtheorem{remark}[thm]{Remark}

\newtheorem{definition}[thm]{Definition}
\newtheorem{example}[thm]{Example}

\newcommand{\ilim}{\mathop{\varprojlim}\limits} 

\newcommand{\Proj}{{\P roj}}

\newcommand{\codim}{{\rm codim}}

\newcommand{\Spec}{{\rm Spec \,}}

\newcommand{\sE}{{\mathcal E}}

\newcommand{\sH}{{\mathcal H}}

\newcommand{\sO}{{\mathcal O}}

\newcommand{\sU}{{\mathcal U}}

\newcommand{\sW}{{\mathcal W}}
\newcommand{\sX}{{\mathcal X}}
\newcommand{\sY}{{\mathcal Y}}
\newcommand{\sZ}{{\mathcal Z}}
\newcommand{\A}{{\mathbb A}}

\newcommand{\G}{{\mathbb G}}

\renewcommand{\P}{{\mathbb P}}
\newcommand{\Q}{{\mathbb Q}}

\newcommand{\Z}{{\mathbb Z}}

\newcommand{\colim}{{\rm colim \,}}
\newcommand{\hocolim}{{\rm hocolim \,}}
\newcommand{\holim}{{\rm holim \,}}
\newcommand{\DM}[2]{\mathbf{DM}_{#2}^{\mathit{eff}}(#1)}

\usepackage{verbatim}

    \newcommand{\spref}[1]{\href{http://stacks.math.columbia.edu/tag/#1}{#1}}
\input{xy}
\xyoption{all}
\begin{document}

\title{On the motivic homotopy type of algebraic stacks}

\author{Neeraj Deshmukh}
\address{Department of Mathematics, Boyd Research and Education Center, University of Georgia, Athens, GA 30602, USA}
\email{neeraj.deshmukh@uga.edu}

\author{Jack Hall}
\address{School of Mathematics \& Statistics\\The University of Melbourne\\Parkville,
  VIC, 3010\\Australia}
\email{jack.hall@unimelb.edu.au}



\date{\today}

\maketitle

\begin{abstract}
  We construct smooth presentations of algebraic stacks that are local epimorphisms in the Morel-Voevodsky $\A^1$-homotopy category. As a consequence, we show that the motive of a smooth stack (in Voevodsky's triangulated category of motives) has many of the same properties as the motive of a smooth scheme.
\end{abstract}
\section{Introduction}

An important technique in motivic homotopy theory of algebraic stacks is reduction to the scheme case by means of homotopical descent. This is possible, for instance, when the stacks in question are Nisnevich locally quotient stacks. The results in \cite{AHR19} (and further generalisation in \cite{ahhlr21}) show that stacks with linearly reductive stabilisers are Nisnevich locally quotient stacks. 

In this note, we establish a certain homotopy descent result for any
algebraic stack.  This allows us to conclude that the various
formalisms of motives and motivic homotopy theory\footnote{In the case
  of stacks one has to distinguish between \textit{genuine} vs
  \textit{Kan-extended} motivic spectra. For instance $K$-theory is
  \textit{genuine} but Chow groups are \textit{Kan-extended}. In this
  note we will be concerned with the latter kind of objects.} produce
the correct results for algebraic stacks. To wit, we will show that
the motive of a smooth algebraic stack has similar properties as the
motive of a smooth scheme. This generalises the results in
\cite{cdh20} that were established for Nisnevich locally quotient
stacks. Another improvement, albeit minor, that we can make is to show
that for algebraic stacks with separated diagonal the existing notions
of stable homotopy category coincide: in \cite{chow21}, it is shown
that the stable motivic homotopy category of \cite{chow21} and the
lisse-extended category of \cite{khan2021generalized} are equivalent
when the stack admits a smooth presentation with a Nisnevich local
section (i.e., the \emph{smooth-Nisnevich coverings} of
\cite{pirisi2018}). In Appendix \ref{App:smooth-nis-covers}, we will
show that all algebraic stacks admit such a presentation (Theorem
\ref{MT}).


\begin{definition}
	A \emph{smooth-Nisnevich covering} of algebraic stacks is a morphism $f:\sY\rightarrow \sX$ such that $f$ is smooth and every morphism $\Spec K \rightarrow \sX$ from the spectrum of a field $K$ lifts to $\sY$. 
\end{definition}

When $f$ is \'{e}tale and $\sY$ and $\sX$ are algebraic spaces, these
are just the standard Nisnevich coverings.

The existence of smooth-Nisnevich coverings for algebraic stacks goes back to
\cite[\S6]{LMB} (quasi-compact and separated diagonal) and
\cite[Appendix B]{cesnavicius15} (diagonal has relatively separated
fibers). If the stack is of finite type over an infinite field with
affine stabilizers, then it was shown in \cite{pirisi2018} that there
exists a smooth-Nisnevich covering of finite type. In the following theorem, we generalize all of these results by
eliminating separation hypotheses for the existence result and
noetherian hypotheses in the boundedness result; we also prove an
analogous result for Deligne--Mumford stacks. 

\begin{theorem}\label{MT}
	If $\mathcal{X}$ is an algebraic stack, then there exists a
	smooth-Nisnevich covering $p\colon X \to \mathcal{X}$, where $X$ is
	a scheme. If $\mathcal{X}$ is quasi-compact and quasi-separated with
	affine stabilizers, then we may take $X$ to be affine. Moreover, if
	$\mathcal{X}$ is Deligne--Mumford, then we may take $p$ to be
	\'etale.
\end{theorem}

It
was brought to our attention during the preparation of our results
that the existence result was recently proved in a similar way in
\cite{chowdhury2024nonrepresentablesixfunctorformalisms}. However, the qcqs and affine stabilzer part of the theorem is new, and we expect it to have applications beyond motivic homotopy theory (see Remark \ref{remark-mathur-kresch} for one such application).




The existence of such covers has implications for motivic homotopy theory of algebraic stacks. More specifically, it shows that the Nisnevich homotopy type of an algebraic stack can be described by a simplicial scheme (or simplicial algebraic space). This makes many homotopical descent arguments accessible for algebraic stacks.

\begin{theorem}\label{corollary-smooth-nisnevich-presentation}
	Let $p:X\rightarrow \sX$ be a smooth-Nisnevich covering over a field $k$. Let $X_{\bullet}$ denote the \v{C}ech nerve of $p$. Then the morphism $p_{\bullet}:X_{\bullet}\rightarrow \sX$ induces an equivalence in the Morel-Voevodsky $\A^1$-homotopy category, $\sH(k)$.
\end{theorem}

A consequence of the above result is that the motive of a smooth algebraic stack continues to enjoy the same properties as the motive of a smooth scheme (see Section \ref{section-applications}). So far this was only known for stacks which satisfied local structure theorems in the sense of \cite{AHR19, ahhlr21} or in a different direction for motives of smooth stacks in the \'{e}tale topology. We describe various results of this kind in the text.

\begin{remark}\label{remark-mathur-kresch}
Another application of Theorem \ref{MT} is the following: In \cite{km22}, the authors use the boundedness result in Theorem \ref{MT} to apply Elkik's approximation technique to the Picard stack and study the following question of Grothendieck: when is the map
\[H^2(X,\G_m)\rightarrow \ilim H^2(X_n,\G_m)\]
injective for a proper morphism $X\rightarrow\Spec A$, where $(A,\mathfrak{m})$ is a complete Noetherian ring and $X_n:= X\times_A \Spec A/\mathfrak{m}^{n+1}$ is the $n$-th infinitesimal thickening.
\end{remark}


\noindent\textbf{Conventions.} We work with algebraic stacks in the sense of \cite{stacks-project}; that is, without separation assumptions unless noted otherwise.\\

\noindent\textbf{Acknowledgements.} Neeraj Deshmukh would like to thank Roberto Pirisi and Tuomas Tajakka for helpful discussions. He also thanks Amit Hogadi,  Siddharth Mathur, Kestutis Cesnavicius and Andrew Kresch for their comments on this note and Utsav Choudhury for teaching him about motives with compact support.\\
The authors would like to thank the referee for their comments.\\
N.D. was supported by the Swiss National Science Foundation (SNF), project 200020\_178729 during the course of this work. He also acknowledges the support of the University of Zurich under the UZH Postdoc Grant, Verf\"{u}gung Nr. FK-22-111.\\
Jack Hall was partially supported
by the Australian Research Council DP210103397 and FT210100405.

\section{Smooth-Nisnevich covers}\label{App:smooth-nis-covers}

Nisnevich coverings for algebraic stacks were discussed in
\cite[\S3]{MR3754421}. The difference between smooth-Nisnevich and
Nisnevich coverings is that the latter morphisms are \'etale and
require a lift of the residual gerbe---not just a rational point. Like
their Nisnevich counterparts, however, smooth-Nisnevich coverings are
clearly stable under composition and base change. We begin with some
examples.
\begin{example}\label{ex:alg-sp}
	Let $f \colon X \to S$ be a smooth morphism of stacks, where $S$ is
	a quasi-separated algebraic space. Then every point $s\in |S|$ has a
	residue field $\kappa(s)$ \cite[Tag
	\spref{03JV}]{stacks-project}. Hence, $f$ is a smooth-Nisnevich
	covering if and only if for every $s \in |S|$, the morphism
	$f_s \colon X\times_S \Spec \kappa(s) \to \Spec \kappa(s)$ admits a
	section. In particular, if $X$ is also an algebraic space, then
	\'etale smooth-Nisnevich coverings are equivalent to Nisnevich
	covers in the sense of \cite[\S3]{MR3754421}. Thus, if
	$\mathcal{X}$ is a quasi-separated algebraic space, then Theorem
	\ref{MT} holds for $\mathcal{X}$ \cite[Prop.~5.7.6]{MR308104}.
\end{example}
\begin{example}\label{ex:gln-nis-sm}
	$\Spec \Z \to B\mathrm{GL}_{n,\Z}$ is a smooth-Nisnevich covering. More
	generally, if $Y \to S$ is a morphism of quasi-separated algebraic
	spaces such that $Y$ admits an $S$-action of $\mathrm{GL}_{n,S}$, then
	Theorem \ref{MT} holds for the quotient stack
	$[Y/\mathrm{GL}_{n,S}]$. Indeed, passing to a Nisnevich covering of $S$, we
	may assume that $S$ is an affine scheme. Then $Y \to [Y/\mathrm{GL}_{n,S}]$
	is a smooth-Nisnevich morphism (it is the base change of
	$\Spec \Z \to B\mathrm{GL}_n$). Passing to a Nisnevich \'etale cover of $Y$
	finishes the argument. This holds more generally for $[Y/G]$, where
	$G \to S$ is a smooth group algebraic space such that
	$\mathrm{H}^1((\Spec K)_{\mathrm{et}},G_{\Spec K})=0$ for all
	$\Spec K \to S$.
\end{example}
\begin{example}\label{ex:quasi-dm}
	If $\mathcal{X}$ is an algebraic stack with quasi-finite diagonal,
	then Theorem \ref{MT} holds for $\mathcal{X}$. We may of course
	assume that $\mathcal{X}$ is quasi-compact. By
	\cite[Thm.~4.1]{MR3754421} there exist morphisms of algebraic
	stacks $V \xrightarrow{v} \mathcal{W} \xrightarrow{w} \mathcal{X}$
	such that $V$ is an affine scheme, $v$ is finite and faithfully flat
	of finite presentation, and $w$ is a Nisnevich \'etale
	covering. Hence, we may replace $\mathcal{X}$ by $\mathcal{W}$ and
	assume that $\mathcal{X}$ admits a finite and faithfully flat cover
	of finite presentation by an affine scheme $V$. By
	\cite[Prop.~4.3(vii)]{GrossRes}, $\mathcal{X}$ has the
	resolution property and so the Totaro--Gross Theorem
	\cite[Thm.~5.4]{GrossRes} implies that
	$\mathcal{X} \simeq [X/\mathrm{GL}_n]$, where $X$ is quasi-affine. Now apply
	Example \ref{ex:gln-nis-sm}.
\end{example}
Let $f \colon X \to S$ be a morphism of algebraic stacks. A
\emph{splitting cover} for $f$ is a sequence of quasi-compact immersions
$Z_j \hookrightarrow S$ for $j=1$, $\dots$, $r$ with
$|S|=\cup_{j=1}^r |Z_j|$ such that
$f_{Z_j} \colon X\times_S Z_j \to Z_j$ admits a section for each $j$.
\begin{remark}\label{R:splitting-sequence}
	If $S$ is quasi-compact, then the existence of a splitting cover
	for $f$ is equivalent to the existence of a \emph{splitting
		sequence} for $f$; that is, there is a filtration
	$\emptyset = S_0 \subseteq S_1 \subseteq \cdots \subseteq S_t = S$
	by quasi-compact open subsets such that
	$f_{S_{i}\setminus S_{i-1}} \colon X \times_S (S_{i}\setminus
	S_{i-1}) \to S_i\setminus S_{i-1}$ admits a section for each $i=1$,
	$\dots$, $t$. The equivalence is due to the close relationship
	between partitions and filtrations in the constructible topology
	\cite[Tag \spref{09XY}]{stacks-project}.
\end{remark}
The
following lemma is a smooth analog of \cite[Prop.~3.3]{MR3754421}.
\begin{lemma}\label{L:split-covering-sm}
	Let $f \colon X \to S$ be a smooth morphism of algebraic stacks. If
	$S$ is a quasi-compact and quasi-separated algebraic space, then $f$ is a
	smooth-Nisnevich covering if and only if $f$ admits a splitting
	cover.
\end{lemma}
\begin{proof}
	Clearly, if $f$ has a splitting cover, then it is a
	smooth-Nisnevich covering. For the other direction: let $s\in |S|$
	be a point. By Example \ref{ex:alg-sp},
	$f_s \colon X\times_S \Spec \kappa(s) \to \Spec \kappa(s)$ admits a
	section. By \cite[Lem.~2.1]{MR3754421} and \cite[Prop.~B.2 \&
	B.3]{RydhApproximation}, there exists an immersion $V_s \hookrightarrow S$
	of finite presentation such that $s\in |V_s|$ and a section to
	$f_{V_s} \colon X\times_S V_s \to V_s$. Since the $V_s$ are
	constructible, $S$ is covered by finitely many and so there are
	$s_j \in |S|$, $j=1$, $\dots$, $r$ such that
	$\cup_{j=1}^{r} |V_{s_j}| = |S|$. 
\end{proof}
The following Proposition leverages splitting coverings to obtain a boundedness result.
\begin{proposition}\label{P:boundedness}
	Consider a cartesian square of algebraic stacks:
	\[
	\xymatrix{X' \ar[r]^{x} \ar[d]_{f'} & X \ar[d]^f \\ S' \ar[r]_s &
		S.}
	\]
	Assume that
	\begin{enumerate}
		\item $f$ and $s$ are smooth-Nisnevich coverings; and
		\item $X$ is a quasi-compact and quasi-separated algebraic space.
	\end{enumerate}
	Then there exists a quasi-compact open $S'' \subseteq S'$ such that
	$S'' \to S$ is a smooth-Nisnevich covering.
\end{proposition}
\begin{proof}
	Since $X$ is a quasi-compact algebraic space and $x$ is a
	smooth-Nisnevich covering, it follows from Lemma
	\ref{L:split-covering-sm} that $x$ admits a splitting cover
	$|X| = \cup_{j=1}^r |W_j|$ with sections
	$t_j \colon W_j \to X\times_S W_j$ to
	$f_{W_j} \colon X\times_S W_j \to W_j$. 
	Since $x$ is quasi-separated and locally of finite presentation, the
	sections $t_j$ are of finite presentation. But the $W_j$ are
	quasi-compact, so the subsets $t_j(W_j) \subseteq X'$ are all
	constructible (Chevalley's Theorem). Write $S' = \cup_{i\in I} S_i$
	as an increasing union of quasi-compact open subsets. Since
	$t_j(W_j)$ is constructible and $f'$ is quasi-compact, for each $j$
	there is an $i_j$ such that
	$t_j(W_j) = f^{-1}(S') \cap t_j(W_j) = f^{-1}(S_{i_j}) \cap
	t_j(W_j)$; in other words, $t_j(W_j) \subseteq
	f^{-1}(S_{i_j})$. Take $v = \max\{i_j\}$ and set $S'' = S_{v}$ and
	$X'' = f'^{-1}(S'') \subseteq X$. Then the induced morphism
	$X'' \to X$ is a smooth-Nisnevich covering (Lemma
	\ref{L:split-covering-sm}). Thus, $X'' \to S$ is a smooth-Nisnevich
	covering and so $S'' \to S$ is a smooth-Nisnevich covering.
\end{proof}

Let $f \colon X \to S$ be a morphism of algebraic stacks. Let
$\underline{\mathrm{HS}}_f \to S$ be the \emph{Hilbert stack} of $f$,
which parameterizes quasi-finite and representable morphisms to $X$
that are proper over the base. If $f$ has affine stabilizers with
quasi-compact and separated diagonal, then
$\underline{\mathrm{HS}}_f \to S$ is a morphism of algebraic stacks
with quasi-affine diagonal, which is locally of finite presentation if
$f$ is so \cite{MR3148551,MR3359029}. As noted in
\cite[Thm.~5.1]{MR3754421}, if we let
$\underline{\mathrm{HS}}_f^{\mathrm{qfb}} \subseteq \underline{\mathrm{HS}}_f$
be the open substack parameterizing those families that are
quasi-finite over $S$, then these existence results are much simpler
to prove. Let
$\underline{\mathrm{HS}}_f^{\mathrm{\acute{e}tb}} \subseteq
\underline{\mathrm{HS}}_f^{\mathrm{qfb}}$ be the substack parameterizing those
families that are \'etale over $S$. If $f$ is a morphism of algebraic spaces, then
$\underline{\mathrm{HS}}_f^{\mathrm{\acute{e}tb}} \simeq [(X/S)^d/S_d]$
\cite[Thm.~5.1]{MR2821738}, where $(X/S)^d$ denotes the $d$-fold fiber
product of $X$ over $S$ and $S_d$ acts on this product via
permutation. If $f$ is a separated morphism of algebraic spaces, then there is an
open subspace
$\underline{\mathrm{Hilb}}_f^{\mathrm{\acute{e}tb}} \subseteq
\underline{\mathrm{HS}}_f^{\mathrm{\acute{e}tb}}$ parameterizing closed immersions
into $X$ and
$\underline{\mathrm{Hilb}}_f^{\mathrm{\acute{e}tb}} \subseteq
\underline{\mathrm{HS}}_f^{\mathrm{\acute{e}tb}}$ corresponds to the open subset of
$[(X/S)^d/S_d]$ that is the complement of the diagonals
\cite[Thm.~5.1]{MR2821738}. The following result summarizes the
relevant results of \cite{MR2821738} and can be viewed as a smooth
variant of \cite[\S\S 5-6]{MR3754421}.
\begin{proposition}\label{P:hs-covering}
	Let $f \colon X \to S$ be a representable smooth covering of algebraic stacks. 
	\begin{enumerate}
		\item \label{PI:hs-covering:HS} $\underline{\mathrm{HS}}_f^{\mathrm{\acute{e}tb}} \to S$ is
		relatively Deligne--Mumford with separated diagonal, and a
		smooth-Nisnevich covering.
		\begin{enumerate}
			\item \label{PI:hs-covering:HS:et} If $f$ is \'etale, then
			$\underline{\mathrm{HS}}_f^{\mathrm{\acute{e}tb}} \to S$ is \'etale.\footnote{If
				$X$ and $S$ are quasi-separated, then $f$ is even a Nisnevich
				covering \cite[Proof of Thm.~4.1]{MR3754421}.}
			\item \label{PI:hs-covering:HS:dm} If $X$ is a Deligne--Mumford
			stack (with separated diagonal), then so is
			$\underline{\mathrm{HS}}_f^{\mathrm{\acute{e}tb}}$.
		\end{enumerate}
		\item \label{PI:hs-covering:hilb} If $f$ is separated, then
		$\underline{\mathrm{Hilb}}_f^{\mathrm{\acute{e}tb}} \to S$ is representable,
		separated, and a smooth-Nisnevich covering.
		\begin{enumerate}
			\item \label{PI:hs-covering:hilb:alg} If $X$ is a (separated) algebraic space, then so is
			$\underline{\mathrm{Hilb}}_f^{\mathrm{\acute{e}tb}}$.
			\item \label{PI:hs-covering:hilb:sch} If $X$ is a (separated) scheme and $f$ is \'etale, then so
			is $\underline{\mathrm{Hilb}}_f^{\mathrm{\acute{e}tb}}$.
		\end{enumerate}
	\end{enumerate}
\end{proposition}
\begin{proof}
	The first parts of \eqref{PI:hs-covering:HS} and
	\eqref{PI:hs-covering:hilb} are smooth-local on $S$ and follow
	trivially from the explicit descriptions above when $f$ is a
	morphism of algebraic spaces. 
	To see that $\underline{\mathrm{HS}}_f^{\mathrm{\acute{e}tb}} \to S$ is a
	smooth-Nisnevich covering, let $s \colon \Spec K \to S$ be a
	morphism, where $K$ is a field. Since $f\colon X \to S$ is a smooth
	covering, there is a finite separable field extension
	$K \subseteq K'$ together with a lift $x \colon \Spec K' \to X$
	of $s$. That is, we have
	$\Spec K' \to X \times_S \Spec K \to \Spec K$, which corresponds
	to a lift of $x$ to $\underline{\mathrm{HS}}_f^{\mathrm{\acute{e}tb}}$. If $X \to S$
	is separated, then $\underline{\mathrm{Hilb}}_f^{\mathrm{\acute{e}tb}} \to S$ is a
	smooth-Nisnevich covering as we may replace $\Spec K'$ to be the
	residue field of the closed point in its image in
	$X\times_S \Spec K$. Next, the universal family gives us a diagram:
	\[
	\xymatrix@R-1pc{\mathcal{Z} \ar[r] \ar[dr] & X\times_S
		\underline{\mathrm{HS}}_f^{\mathrm{\acute{e}tb}} \ar[r] \ar[d] & X \ar[d] \\ &
		\underline{\mathrm{HS}}_f^{\mathrm{\acute{e}tb}} \ar[r] & S.}
	\]
	If $X$ is Deligne--Mumford (with separated diagonal), then
	$X\times_S \underline{\mathrm{HS}}_f^{\mathrm{\acute{e}tb}}$ is Deligne--Mumford
	(with separated diagonal). The map
	$\mathcal{Z} \to X\times_S \underline{\mathrm{HS}}_f^{\mathrm{\acute{e}tb}}$ is
	quasi-finite, separated and representable and so $\mathcal{Z}$ is
	Deligne--Mumford (with separated diagonal). But
	$\mathcal{Z} \to \underline{\mathrm{HS}}_f^{\mathrm{\acute{e}tb}}$ is finite \'etale
	and surjective and so $\underline{\mathrm{HS}}_f^{\mathrm{\acute{e}tb}}$ is
	Deligne--Mumford (with separated diagonal). Finally,
	\eqref{PI:hs-covering:hilb:alg} follows from
	\cite[Thm.~4.1]{MR2821738} and \eqref{PI:hs-covering:hilb:sch}
	follows from \cite[Rem.~2.3 \& Thm.~2.4]{MR2821738}.
\end{proof}
\begin{proof}[Proof of Theorem \ref{MT}]
	For the existence: let $v \colon V \to \mathcal{X}$ be a smooth cover,
	where $V$ is a scheme (take $v$ to be \'etale if $\mathcal{X}$ is
	Deligne--Mumford). Then $v$ is representable and let
	$w \colon W=\underline{\mathrm{HS}}_v^{\mathrm{\acute{e}tb}} \to \mathcal{X}$ be the
	induced morphism. Proposition
	\ref{P:hs-covering}\eqref{PI:hs-covering:HS} implies that $w$ is a
	smooth-Nisnevich covering and $W$ is a Deligne--Mumford stack with
	separated diagonal (if $\mathcal{X}$ is a Deligne--Mumford stack, $w$
	is even \'etale). Replacing $\mathcal{X}$ by $W$ we are reduced to the
	situation where $\mathcal{X}$ is Deligne--Mumford with separated
	diagonal. In this case, $v$ is separated, representable, and
	\'etale. Then Proposition
	\ref{P:hs-covering}\eqref{PI:hs-covering:hilb} implies that
	$p \colon X = \underline{\mathrm{Hilb}}^{\mathrm{\acute{e}tb}}_v \to \mathcal{X}$ is
	an \'etale smooth-Nisnevich covering and $X$ is a separated scheme.
	
	For the boundedness: by
	\cite[Prop.~2.6(i)]{MR3436239},
	$|\mathcal{X}|$ admits a finite partition
	$\amalg_{j=1}^r|\mathcal{W}_j|$, where
	$\mathcal{W}_j = [W_j/\mathrm{GL}_{n_j}]$ and $W_j$ is a quasi-affine
	scheme. For each $j$ form the following cartesian diagram:
	\[
	\xymatrix@R-1pc{W_j' \ar[r]^{p_j'} \ar[d]_{q_j'} & W_j \ar[d]^{q_j} \\
		X\times_{\mathcal{X}} \mathcal{W}_j \ar[r]_{p_j} & \mathcal{W}_j.}
	\]
	Then $p_j$ is a smooth-Nisnevich covering and so too is $q_j$
	(Example \ref{ex:gln-nis-sm}). By Proposition \ref{P:boundedness}, it
	follows that there is a quasi-compact open
	$W''_j \subseteq X\times_{\mathcal{X}} \mathcal{W}_j$, which is also a
	scheme, such that $W''_j \to \mathcal{W}_j$ is a
	smooth-Nisnevich cover. Now take a quasi-compact open
	$X'' \subseteq X$ such that
	$W''_j \subseteq X'' \times_{\mathcal{X}} \mathcal{W}_j$ for all
	$j$. Then $X'' \to \mathcal{X}$ is a smooth-Nisnevich cover, which
	proves the claim.
\end{proof}

\section{Applications to Motivic homotopy theory of algebraic stack}\label{section-applications}
Let $k$ be a field. Let $Sm/k$ denote the category of smooth schemes over $k$. Let $\Delta^{op}PSh(Sm/k)$ be the category of simplicial presheaves on $Sm/k$. Then $\Delta^{op}PSh(Sm/k)$ has a local model structure with respect to the Nisnevich topology (see \cite{Jar}). A morphism $f:X\rightarrow Y$ in $\Delta^{op}PSh(Sm/k)$ is a weak equivalence if the induced morphisms on stalks (for the Nisnevich topology) are weak equivalences of simplicial sets. Cofibrations are monomorphisms, and fibrations are characterised by the right lifting property. 

We Bousfield localise this model structure with respect to the class of maps $X\times \A^1\rightarrow X$ (see \cite[3.2]{MV}). The resulting model structure is called the Nisnevich motivic model structure. Denote by $\mathcal{H}(k)$ the resulting homotopy category. This is the (unstable) $\A^1$-homotopy category for smooth schemes over $k$.

Let us also rapidly recall the definition of Voevodsky's triangulated category of motives $\DM{k,\Z}{}$. Let $Cor_k$ denote the category of finite correspondences whose objects are smooth separated schemes over $k$. For any two $X,Y$, the morphisms of $Cor_k$ are given by $Cor(X,Y)$ which is the free abelian group generated by irreducible closed subschemes $W\subset X\times Y$ that are finite and surjective over $X$. An additive functor $F:Cor_k^{op}\rightarrow \mathbf{Ab}$ is called a presheaf with transfers. Let $PST(k,\Z)$ denote the category of presheaves with transfers. For any smooth scheme $X$, let $\Z_{tr}(X)$ be the presheaf with tranfers which on any smooth scheme $Y$ is defined as
\[\Z_{tr}(X)(Y):=Cor(Y,X)\]
Let $K(PST(k,\Z))$ denote the category of complexes of presheaves with transfers. The category $K(PST(k,\Z))$ also has a Nisnevich motivic model structure which is defined analogously as in the case of $\Delta^{op}PSh(Sm/k)$. We denote the associated homotopy category by $\DM{k,\Z}{}$. This is Voevodsky's triangulated category of mixed motives in the Nisnevich topology (for details, see \cite{MWV}).

We now have the following trivial observation.
\begin{lemma}\label{lemma-lift-hensel-local-points}
  Let $p \colon X\rightarrow\sX$ be a smooth-Nisnevich covering of algebraic stacks. If
  $\Spec \sO$ is the spectrum of a Henselian local ring $\sO$, then
  every morphism $\Spec \sO\rightarrow \sX$ factors through $p$.
\end{lemma}
\begin{proof}
	Let $K$ be the residue field of $\sO$ and $X_\sO:= X\times_\sX \Spec \sO$ denote the base change of $X$ to $\Spec \sO$. Then the composite $\Spec K \rightarrow \Spec \sO \rightarrow \sX$ lifts to a morphism $\Spec K \rightarrow X$. This also gives us a morphism $\Spec K \rightarrow X_\sO$. Now, as $X_\sO \rightarrow \Spec \sO$ is smooth, we have a surjection of sets $X_\sO(\sO)\twoheadrightarrow X_\sO(K)$. Thus, we have a section $\Spec \sO \rightarrow X_\sO$. Composing with the projection map $X_\sO \rightarrow X$ gives us the desired lift.
\end{proof}

\begin{proposition}
	Let $\sX$ be an algebraic stack that is locally of finite type over a field $k$. Let $X\rightarrow \sX$ be a smooth-Nisnevich covering. Then the associated \v{C}ech hypercover $X_{\bullet}$ is weakly equivalent to $\sX$ in the Nisnevich local model structure on the category $\Delta^{op}PSh (Sm/k)$ of simplicial presheaves.
\end{proposition}
\begin{proof}
	By Lemma \ref{lemma-lift-hensel-local-points}, we know that hensel local points lift along smooth-Nisnevich coverings. Adapting the proof of \cite[Theorem 1.2]{cdh20} gives the result.
\end{proof}

The proof of Theorem \ref{corollary-smooth-nisnevich-presentation} follows from $\A^1$-localising the above proposition.

\begin{proof}[Proof of Theorem \ref{corollary-smooth-nisnevich-presentation}]
	As $\A^1$-localisation preserves simplicial equivalences, the result follows from the previous proposition.
\end{proof}

As a consequence, we have the following result.

\begin{corollary}
	The weak equivalence above induces an equivalence of motives $M(X_{\bullet})\simeq M(\sX)$ in $\DM{k,\Z}{}$.
\end{corollary}
\begin{proof}
	Use the functor $M: \sH_{\bullet}(k)\rightarrow \DM{k,\Z}{}$ (see \cite{cdh20} for details).
\end{proof}

Combining these results with Theorem \ref{MT}, we obtain motivic
results for smooth algebraic stacks.  The proofs are exactly as in
\cite{cdh20} after replacing $\mathrm{GL}_n$-presentations with
smooth-Nisnevich presentations.
\begin{theorem}
	Let $\sX$ be an algebraic stack that is smooth over a field $k$. Then its motive $M(\sX)\in \DM{k,\Z}{}$ satisfies the following properties:
	\begin{enumerate}
		\item $M(\sX)$ satisfies Nisnevich descent.
		\item (Projective bundle formula) For any vector bundle $\sE$ of rank $n+1$ over $\sX$, we have a canonical isomorphism
		\[M(\Proj(\sE))\simeq \bigoplus_{i=0}^n M(\sX)(i)[2i]\]
		\item (Blow-up formula) For $\sZ \subset \sX$ a smooth closed substack of pure codimension $c$ we have
		\[M(Bl_{\sZ}(\sX))\simeq M(\sX)\oplus_{i=0}^{c-1} M(\sZ)(i)[2i]\]
		\item (Gysin triangle) For $\sZ \subset \sX$ a smooth closed substack of codimension c, we have a Gysin triangle:
		\[M(\sX \setminus \sZ)\rightarrow M(\sX)\rightarrow M(\sZ)(c)[2c] \rightarrow M(\sX \setminus \sZ)[1].\]
	\end{enumerate}
\end{theorem}
\begin{proof}
	The proofs in \cite[\S 4]{cdh20} go through using a smooth-Nisnevich covering $X\rightarrow\sX$.
\end{proof}

\begin{remark}
	Note that the transfer functor $M: \sH_{\bullet}(k)\rightarrow \DM{k,\Z}{}$ allows us to define the motive of \textit{any} smooth stack. What is apriori unclear is whether this motive has any properties (good or bad). This is why in \cite{cdh20}, the authors work with stacks that are Nisnevich locally quotient stacks in order to use homotopical descent to prove the above formulae for the motive. With Theorem \ref{corollary-smooth-nisnevich-presentation}, those arguments now work for all smooth stacks.
\end{remark}


\subsection{Motivic cohomology as hypercohomology}
We briefly recall the definitions of smooth-Nisnevich and
smooth-Zariski sites of an algebraic stack.
\begin{definition}
	Let $\sX$ be an algebraic stack. The smooth-Nisnevich (resp. smooth-Zariski) site of $\sX$, denoted by $\sX_{\text{lis-nis}}$ (resp.\ $\sX_{\text{lis-zar}}$) is the category whose objects consist of pairs $(U,p)$ where $p$ is a smooth morphism $p:U\rightarrow \sX$  from an algebraic space $U$ and coverings in $\sX_{\text{lis-nis}}$ are given by Nisnevich coverings (resp. Zariski coverings) of algebraic spaces.
\end{definition}

The following result states that motivic cohomology can be computed as
hypercohomology of the motivic complexes $\Z(i)$ on the
smooth-Nisnevich or smooth-Zariski (in the Deligne--Mumford case) site
of the stack.

\begin{proposition}\label{proposition-hypercohomology}
	Let $\sX$ be an algebraic stack that is smooth over a field $k$. The motivic cohomology of $\sX$ agrees with the hypercohomology of the motivic complexes $\Z(j)$ on $\sX_{\text{lis-nis}}$. That is,
	\[Ext^i_{D(\sX)}(\Z,\Z(j)|_\sX)\simeq
          Ext^i_{DA(\sX)}(\Z,\Z(j)|_\sX)\simeq
          Hom_{\DM{k,\Z}{}}(M(\sX),\Z(j)[i]),\] where $\Z$ denotes the
        constant sheaf $\Z$ on the $\sX_{\text{lis-nis}}$. In
        addition, if $\sX$ is separated and Deligne--Mumford with
        schematic coarse space, then the motivic cohomology of $\sX$
        agrees with the hypercohomology of the motivic complexes
        $\Z(j)$ on $\sX_{\text{lis-zar}}$.
\end{proposition}
\begin{proof}
  See \cite[Theorem 5.2]{cdh20} for the smooth-Nisnevich case. In the
  Deligne--Mumford case, it follows from \cite{KreschGeometry}, where it is shown that such stacks are Zariski-locally quotient stacks.
\end{proof}






\subsection{Motivic cohomology with finite coefficients} We also have the following comparison theorem relating motivic cohomology with $\Z/n\Z$ to \'{e}tale cohomology with $\mu_n$-coefficients.
\begin{corollary}\label{corollary-bl}
	Let $\sX$ be an algebraic stack that is smooth over a field $k$. Then the homomorphisms
	\[H^{p,q}_M (\sX,\Z/n\Z)\rightarrow H^p_{\acute{e}t}(\sX,\mu_n^{\otimes q}),\]
	are isomorphisms for $p\leq q$ and monomorphisms for $p=q+1$.
\end{corollary}
\begin{proof}
	See \cite[Corollary 5.3]{cdh20}.
\end{proof}

\subsection{Motives with compact support}

In this subsection, we will define the notion of a motive with compact support for smooth algebraic stacks with affine stabilizers over a field $k$. The definition is very similar to Totaro's construction for quotient stacks using approximation by vector bundles \cite{TotaroChow}, except that since we cannot approximate non-quotient stacks by algebraic spaces, we are forced to work with homotopy limits. We will use use the definition of dimension for algebraic stacks as in \cite[Tag \spref{0AFL}]{stacks-project}

\begin{definition}
	Let $\sX$ be a connected algebraic stack of finite type over a field $k$ with affine stabilizers. Consider a smooth-Nisnevich covering $p:X\rightarrow\sX$, where $X$ is of finite type over $k$ (Theorem \ref{MT}) and let $n$ denote the relative dimension of $p: X\rightarrow\sX$. Let  $p_{\bullet}:X_{\bullet}\rightarrow \sX$ be the \v{C}ech nerve of $p$. Then we define the \textit{compactly supported motive} of $\sX$ as
	\[M^c(\sX):= \holim\!_i M^c(X_i)(-(i+1)n^2)[-2n^2i].\]
\end{definition}

This definition is a bit weird looking, but it has the following pleasant feature, which also shows that the definition is independent of choices for smooth algebraic stacks.

\begin{theorem}[Poincar\'{e} Duality]
	Let $\sX$ be a smooth algebraic stack of dimension $d$. Then we have an isomorphism, 
	\[M^c(\sX) = M(\sX)^\vee(d)[2d].\]
\end{theorem}

\begin{proof}
	The proof is a straightforward manipulation of the definition. Tensoring by the dualizable object $\Z(-d)[-2d]$ we get:
	\begin{align*}
	M^c(\sX)\otimes \Z(-d)[-2d] &= (\holim\!_i M^c(X_i)(-(i+1)n^2)[-2n^2i])\otimes \Z(-d)[-2d]\\
	&= \holim\!_i M^c(X_i)(-(i+1)n^2-d)[-2n^2i-2d]\\
	&= \holim\!_i M(X_i)^\vee.
	\end{align*}
	where the last step follows from the fact that each $X_i$ is smooth. Further, since $M(\sX)\simeq \hocolim\!_iM(X_i)$, taking dual we get that $M(\sX)^\vee=\holim\!_i M(X_i)^\vee$. Thus, we have proved that
	\[M^c(\sX)(-d)[2d]\simeq M(\sX)^\vee.\]
	Tensoring by $\otimes \Z(d)[2d]$, gives the required expression.
\end{proof}

\noindent In the following subsection, we note an application to stable motivic homotopy theory of stacks. We will not explain any details to prevent drowning the reader in $\infty$-categories. Our intention is to simply indicate how the geometric content of Theorem \ref{MT} can be applied in the world of motivic homotopy theory. The interested reader may consult \cite{chow21} or \cite{khan2021generalized} for further details.
\subsection{The stable motivic homotopy category of a stack}
Recently, two different definitions of the stable motivic homotopy category for stacks have appeared in literature. The limit-extended category $SH_\lhd(\sX)$ in \cite{khan2021generalized} and the category $SH_{ext}^{\otimes}(\sX)$ in \cite{chow21} that is extended from schemes to stacks having Nisnevich local sections. In \cite{chow21} it is proved that the two definition are equivalent
whenever the stack admits a smooth-Nisnevich covering. As coverings such always exist for algebraic stacks by Theorem \ref{MT}, we have a strengthening of \cite[Corollary 2.5.4]{chow21}.

\begin{definition}
	Let $\sX$ be an algebraic stack. A smooth presentation $X\rightarrow\sX$ is said to admit \textit{Nisnevich local sections} if, for any scheme $Y$ and a morphism $Y\rightarrow\sX$, there exists a Nisnevich covering $Y'\rightarrow Y$ with a section $Y'\rightarrow Y\times_\sX X$. Thus, we have the following commutative diagram,
	\begin{center}
		\begin{tikzcd}
			& Y\times_\sX X\arrow[d]\arrow[r] & X\arrow[d]\\
			Y' \arrow[ur, dotted]\arrow[r] & Y\arrow[r] & \sX
		\end{tikzcd}
	\end{center}
\end{definition}

The following lemma shows that the above definition is equivalent to the notion of a smooth-Nisnevich covering.

\begin{lemma}
	Let $\sX$ be an algebraic stack. A smooth presentation $f: X\rightarrow \sX$ is admits Nisnevich local sections if and only if it is a smooth-Nisnevich covering.
\end{lemma}
\begin{proof}
	The only if part is clear. To see the if part, observe that if $X\rightarrow\sX$ is a smooth-Nisnevich covering, then by Cor 2.7, for any Henselian local ring $\sO$ with a map $\Spec \sO\rightarrow \sX$ we have a lift $\Spec\sO \rightarrow X$. Thus, if $V\rightarrow\sX$ is a morphism and $\sO_{v,V}$ is the Henselian local ring at a point $v\in V$, we have a lift $\Spec\sO_{v,V}\rightarrow V\times_\sX X$. By standard limit arguments [EGA], there exists a Nisnevich neighbourhood $V'$ of $v\in V$, and a section $V'\rightarrow V\times_\sX X$.
\end{proof}

\begin{corollary}
	Let $\sX$ be a quasi-separated algebraic stack. Then $SH_\lhd (\sX)\simeq SH_{ext}^{\otimes}(\sX)$, whenever they are defined.
\end{corollary}
\begin{proof}
	The proof in \cite[Corollary 2.5.4]{chow21} works verbatim using a smooth-Nisnevich covering $X\rightarrow \sX$.
\end{proof}

\begin{remark}
	Theorem \ref{MT} also improves \cite[Corollary 12.28]{khan2021generalized} in a similar fashion.
\end{remark}

\appendix

\bibliography{mybib-JACK-EDIT.bib}

\providecommand{\bysame}{\leavevmode\hbox to3em{\hrulefill}\thinspace}
\providecommand{\MR}{\relax\ifhmode\unskip\space\fi MR }
\providecommand{\MRhref}[2]{%
  \href{http://www.ams.org/mathscinet-getitem?mr=#1}{#2}
}
\providecommand{\href}[2]{#2}
\begin{thebibliography}{AHHLR24}

\bibitem[AHHLR24]{ahhlr21}
Jarod Alper, Jack Hall, Daniel Halpern-Leistner, and David Rydh, \emph{Artin
  algebraization for pairs with applications to the local structure of stacks
  and {F}errand pushouts}, Forum Math. Sigma \textbf{12} (2024), Paper No. e20,
  25. \MR{4699880}

\bibitem[AHR19]{AHR19}
Jarod Alper, Jack Hall, and David Rydh, \emph{The étale local structure of
  algebraic stacks}, 2019.

\bibitem[CD24]{chowdhury2024nonrepresentablesixfunctorformalisms}
Chirantan Chowdhury and Alessandro D'Angelo, \emph{Non-representable
  six-functor formalisms}, 2024.

\bibitem[CDH20]{cdh20}
Utsav Choudhury, Neeraj Deshmukh, and Amit Hogadi, \emph{The nisnevich motive
  of an algebraic stack}, 2020.

\bibitem[{\v{C}}es15]{cesnavicius15}
K{\k{e}}stutis {\v{C}}esnavi{\v{c}}ius, \emph{Topology on cohomology of local
  fields}, Forum Math. Sigma \textbf{3} (2015), 55 (English), Id/No e16.

\bibitem[Cho21]{chow21}
Chirantan Chowdhury, \emph{Motivic homotopy theory of algebraic stacks}, 2021.

\bibitem[Gro17]{GrossRes}
Philipp Gross, \emph{Tensor generators on schemes and stacks}, Algebr. Geom.
  \textbf{4} (2017), no.~4, 501--522. \MR{3683505}

\bibitem[HR14]{MR3148551}
Jack Hall and David Rydh, \emph{The {H}ilbert stack}, Adv. Math. \textbf{253}
  (2014), 194--233. \MR{3148551}

\bibitem[HR15a]{MR3436239}
\bysame, \emph{Algebraic groups and compact generation of their derived
  categories of representations}, Indiana Univ. Math. J. \textbf{64} (2015),
  no.~6, 1903--1923. \MR{3436239}

\bibitem[HR15b]{MR3359029}
\bysame, \emph{General {H}ilbert stacks and {Q}uot schemes}, Michigan Math. J.
  \textbf{64} (2015), no.~2, 335--347. \MR{3359029}

\bibitem[HR18]{MR3754421}
\bysame, \emph{Addendum to ``\'{E}tale d\'{e}vissage, descent and pushouts of
  stacks'' [{J}. {A}lgebra 331 (1) (2011) 194--223] [ {MR}2774654]}, J. Algebra
  \textbf{498} (2018), 398--412. \MR{3754421}

\bibitem[Jar87]{Jar}
J.~F. Jardine, \emph{Simplicial presheaves}, J. Pure Appl. Algebra \textbf{47}
  (1987), no.~1, 35--87. \MR{906403}

\bibitem[KM22]{km22}
Andrew Kresch and Siddharth Mathur, \emph{Formal gaga for gerbes}, 2022.

\bibitem[KR21]{khan2021generalized}
Adeel~A. Khan and Charanya Ravi, \emph{Generalized cohomology theories for
  algebraic stacks}, 2021.

\bibitem[Kre09]{KreschGeometry}
Andrew Kresch, \emph{On the geometry of {D}eligne-{M}umford stacks}, Algebraic
  geometry---{S}eattle 2005. {P}art 1, Proc. Sympos. Pure Math., vol.~80, Amer.
  Math. Soc., Providence, RI, 2009, pp.~259--271. \MR{2483938}

\bibitem[LMB00]{LMB}
G\'{e}rard Laumon and Laurent Moret-Bailly, \emph{Champs alg\'{e}briques},
  Ergebnisse der Mathematik und ihrer Grenzgebiete. 3. Folge. A Series of
  Modern Surveys in Mathematics [Results in Mathematics and Related Areas. 3rd
  Series. A Series of Modern Surveys in Mathematics], vol.~39, Springer-Verlag,
  Berlin, 2000. \MR{1771927}

\bibitem[MV99]{MV}
Fabien Morel and Vladimir Voevodsky, \emph{$\mathbb{A}^1$-homotopy theory of
  schemes}, Publications Math{\'e}matiques de l'Institut des Hautes {\'E}tudes
  Scientifiques \textbf{90} (1999), no.~1, 45--143.

\bibitem[MVW06]{MWV}
Carlo Mazza, Vladimir Voevodsky, and Charles Weibel, \emph{Lecture notes on
  motivic cohomology}, Clay Mathematics Monographs, vol.~2, American
  Mathematical Society, Providence, RI; Clay Mathematics Institute, Cambridge,
  MA, 2006. \MR{2242284}

\bibitem[Pir18]{pirisi2018}
Roberto Pirisi, \emph{Cohomological invariants of algebraic stacks}, Trans.
  Amer. Math. Soc. \textbf{370} (2018), no.~3, 1885--1906. \MR{3739195}

\bibitem[RG71]{MR308104}
Michel Raynaud and Laurent Gruson, \emph{Crit\`eres de platitude et de
  projectivit\'{e}. {T}echniques de ``platification'' d'un module}, Invent.
  Math. \textbf{13} (1971), 1--89. \MR{308104}

\bibitem[Ryd11]{MR2821738}
David Rydh, \emph{Representability of {H}ilbert schemes and {H}ilbert stacks of
  points}, Comm. Algebra \textbf{39} (2011), no.~7, 2632--2646. \MR{2821738}

\bibitem[Ryd15]{RydhApproximation}
\bysame, \emph{Noetherian approximation of algebraic spaces and stacks}, J.
  Algebra \textbf{422} (2015), 105--147. \MR{3272071}

\bibitem[{Sta}18]{stacks-project}
The {Stacks Project Authors}, \emph{\textit{Stacks Project}},
  \url{https://stacks.math.columbia.edu}, 2018.

\bibitem[Tot99]{TotaroChow}
Burt Totaro, \emph{The {C}how ring of a classifying space}, Algebraic
  {$K$}-theory ({S}eattle, {WA}, 1997), Proc. Sympos. Pure Math., vol.~67,
  Amer. Math. Soc., Providence, RI, 1999, pp.~249--281. \MR{1743244}

\end{thebibliography}
\bibliographystyle{amsalpha}

\end{document}